\documentclass[12pt,a4paper]{amsart}
\usepackage{graphics}
\usepackage{epsfig}
\usepackage{graphicx}
\usepackage[all]{xy}
\theoremstyle{plain}
\usepackage{amssymb}
\usepackage{tikz-cd}
\usepackage[enableskew]{youngtab}
\advance\hoffset-20mm \advance\textwidth40mm

\newtheorem{theorem}{Theorem}
\numberwithin{theorem}{section}
\newtheorem{lemma}[theorem]{Lemma}
\newtheorem{proposition}[theorem]{Proposition}

\theoremstyle{definition}
\newtheorem{definition}[theorem]{Definition}

\newtheorem{example}[theorem]{Example}

\theoremstyle{empty}
\newtheorem{thm}{Theorem}

\begin{document}
\sloppy

\title[Hilbert-Samuel sequences of homogeneous finite type]
{Hilbert-Samuel sequences of homogeneous finite type}
\author[Konstantin Loginov]{Konstantin Loginov}
\address{Faculty of Mathematics, National Research University Higher School of Economics.
119048 Moscow, Usacheva str., 6  }
\email{kostyaloginov@gmail.com}
\begin{abstract}

This paper deals with the problem of the classification of the local graded Artinian quotients $\mathbb{K}[x, y]\slash I$ where $\mathbb{K}$ is an algebraically closed field of characteristic $0$. They have a natural invariant called Hilbert-Samuel sequence. We say that a Hilbert-Samuel sequence is of homogeneous finite type, if it is the Hilbert-Samuel sequence of a finite number of isomorphism classes of graded local algebras. We give the list of all the Hilbert-Samuel sequences of homogeneous finite type in the case of algebras generated by $2$ elements of degree $1$.
\end{abstract}
\maketitle
%
%%%%%%%%%%%%%%%%%%%%%%%%%%%%%%%%%%%%%%%%%%%%%%%%%%%%%%%%%%%%%%%%%%%%%%%%
%
\section*{Introduction}
Fix an algebraically closed field $\mathbb{K}$ of characteristic $0$. All the algebras in this paper are associative, commutative, and over $\mathbb{K}$. We study the Artinian, or, equivalently, finite-dimensional, case. We consider the case of local algebras, because each Artinian algebra is a direct sum of local ones. Each local Artinian algebra $R$ has a natural invariant $\mathrm{T}(R) = ( t_k )_{k\geq0}$ where  $t_k=\dim_{\mathbb{K}} \mathfrak{m}^k\slash \mathfrak{m}^{k+1}$ and $\mathfrak{m}$ is the maximal ideal of $R$. The sequence $\mathrm{T}(R)$ is called \textit{Hilbert-Samuel sequence of} $R$. It is meaningful to classify local algebras according to their Hilbert-Samuel sequences. The complete classification is known for algebras of dimension not greater than $6$. Necessary references are listed below. 

We give a partial solution to the following problem, proposed by Ivan Arzhantsev: find all the Hilbert-Samuel sequences for which the number of pairwise non-isomorphic algebras is finite. We deal with graded algebras that can be generated by $2$ elements of degree $1$. In terms of the Hilbert-Samuel sequence this is equivalent to $t_1=2$. These algebras can be realized as quotients of the polynomial ring $\mathbb{K}[x, y]$ by a homogeneous ideal $I$ of finite colength. It is known that the Hilbert-Samuel sequence of such algebras has the form $\mathrm{T}=(1, 2, \dots, n, t_n, \dots, t_{n+k})$ where $t_i=i+1$ for $0 \leq i \leq n-1$ and $n=t_{n-1} \geq t_n \geq \dots \geq t_{n+k}$ ([6], Lemma 1.3). Under these assumptions, a complete solution to the problem is obtained: in Theorem $1$ we give the complete list of the Hilbert-Samuel sequences $\mathrm{T}$ for which there are finitely many classes of isomorphism of graded quotients $\mathbb{K}[x, y] \slash I$ with Hilbert-Samuel sequence $\mathrm{T}$. We call them \textit{sequences of homogeneous finite type}. As a consequence, for any other sequence there are infinitely many non-isomorphic (graded) algebras with $t_1=2$. We say that these sequences are of infinite type. The result is as follows:

\begin{thm}

A Hilbert-Samuel sequence with $t_1 = 2$ is of homogeneous finite type if and only if it is in the following table:

\begin{tabular}{ | p{0.9cm} | p{7.5cm} | p{5.3cm} | p{1.6cm} |}

\hline
& Sequence & Restrictions  & $\mathrm{dim}\ G_{\mathrm{T}}$ \\  \hline
   
$\mathrm{T}_1$ & $(1, 2, \dots, n-1, n)$ & $n\geq 2$ & $0$ \\ \hline
   
$\mathrm{T}_2$ & $(1, 2, 1)$ & & $2$ \\  \hline
   
$\mathrm{T}_3$ & $(1, 2, 3, 1)$ & & $3$ \\  \hline
   
$\mathrm{T}_4$ & $(1, 2, 3, 2, 1, \dots, 1)$ & \parbox[t]{8cm}{ $t_3=2,\ t_4=\dots=t_{4+k}=1,
\\
k\geq 1$ } & $3$ \\  \hline
   
$\mathrm{T}_5$ & $(1, 2, \dots, n-1, n, 1, 1, \dots, 1)$ & 
\parbox[t]{8cm}{ $t_n=\dots=t_{n+k}=1,
\\
n\geq2,\ k\geq 1$ } & $1$ \\  \hline

$\mathrm{T}_6$ & $(1, 2, \dots, n-1, n, 2, 2, \dots, 2)$ & \parbox[t]{8cm}{ $t_n=\dots=t_{n+k}=2,
\\
n \geq 1,\ k\geq 1$ } & $2$ \\ \hline
   
$\mathrm{T}_7$ & $(1, 2, \dots, n-1, n, 2, 2, \dots, 2, 1, \dots, 1)$ & \parbox[t]{8cm}{
$t_n=\dots=t_{n+k}=2, 
\\
t_{n+k+1}=\dots=t_{n+k+l}=1,
\\
n \geq 1,\ k\geq 1,\ l\geq 1$ } & $2$ if $l>1$ and $3$ if 
$l = 1$ \\  \hline

$\mathrm{T}_8$ & $(1, 2, 3, \dots, n-1, n, 3, 3, \dots, 3)$ & \parbox[t]{8cm}{ $t_n=\dots=t_{n+k}=3,
\\
n \geq 2,\ k\geq 1$ } & $3$ \\  \hline
   
$\mathrm{T}_9$ & $(1, 2, 3, \dots, n-1, n, 3, 3, \dots, 3, 1, 1, \dots, 1)$ & \parbox[t]{8cm}{ $t_n=\dots=t_{n+k}=3, 
\\
t_{n+k+1}=\dots=t_{n+k+l}=1, 
\\
n \geq 2,\ k\geq 1,\ l\geq 2$ } & $3$ \\  \hline
   
$\mathrm{T}_{10}$ & $(1, 2, 3, \dots, n-1, n, 3, 3, \dots, 3, 2, 2, \dots, 2)$ & \parbox[t]{8cm}{
$t_n=\dots=t_{n+k}=3,
\\
t_{n+k+1}=\dots=t_{n+k+l}=2,
\\
n \geq 2,\ k\geq 1,\ l\geq 2$ } & $3$ \\  \hline
   
$\mathrm{T}_{11}$ & 
\parbox[t]{8cm}{ $ (1, 2, 3, \dots, n-1, n, 3, 3, \dots, 3, 2, 2, \dots, 2, $ \\ $1, 1, \dots, 1) $ } & \parbox[t]{8cm}{
$t_n=\dots=t_{n+k}=3,
\\
t_{n+k+1}=\dots = t_{n+k+l}=2,
\\
t_{n+k+l+1}=\dots = t_{n+k+l+s}=1,
\\
n\geq 2,\ k\geq 1,\ l\geq 2,\ s\geq 2$ } & $3$ \\  \hline

\end{tabular}

\

Table 1. 

\end{thm}

\

Our main tool in proving Theorem $1$ is a result of A. Iarrobino, see [6]. Following his notation, let $\mathrm{Hilb}^N\ \mathbb{K}[[x, y]]$ be the Hilbert scheme (with reduced scheme structure) parametrizing ideals of colength $N$ in $\mathbb{K}[[x, y]]$. Let $G_{\mathrm{T}}$ be the subscheme of $\mathrm{Hilb}^N\  \mathbb{K}[[x, y]]$ parametrizing the homogenenous ideals of the type $\mathrm{T}$, that is having the Hilbert-Samuel sequence $\mathrm{T} = (1, 2, \dots, n, t_n, \dots, t_{n+k})$ with $n = t_{n-1} \geq t_n \geq \dots \geq t_{n+k}$. Let us denote $e_j = t_{j-1} - t_j$ if $j \geq n$ and $e_j = 0$ otherwise. These numbers $e_j$ are called \textit{jump indices} of $\mathrm{T}$. 

\begin{thm} $[6$, Theorem $2.12]$
In the notation as above $$\mathrm{dim}\ G_{\mathrm{T}} = \sum_{j\geq n} (e_j + 1) e_{j+1}$$
\end{thm}

In Proposition $2.1$ we show that the sequences listed in Table $1$ are precisely those for which $\mathrm{dim}\ \mathrm{G}_{\mathrm{T}} \leq 3$. The number $3$ here comes from the dimension of $\mathrm{PGL}(2)$. The reason for this is as follows: the group $\mathrm{PGL} (2)$ acts on $\mathrm{G}_{\mathrm{T}}$ identifying isomorphic algebras. As a consequence, if $\mathrm{dim}\  \mathrm{G}_{\mathrm{T}} > \mathrm{dim} \ \mathrm{PGL}(2)$ then the number of orbits is infinite and hence there are infinitely many non-isomorphic algebras, so the sequence $\mathrm{T}$ is of infinite type. It remains to show that the sequences listed above are of homogeneous finite type. This is done in Propositions $3.1 - 3.7$.    

Now we present some motivation for the problem. The classification of local finite-dimensional algebras is closely related to the  classification of pairs of commuting nilpotent operators (see [3]) and to the study of punctual Hilbert scheme $\mathrm{Hilb}^N\ \mathbb{K}[[x, y]]$, see $[3]$, $[6]$, $[7]$ and $[8]$.

Local algebras arise naturally in the theory of algebraic transformation groups, as shown by Hassett and Tschinkel. Namely, in Proposition $2.15$ and Section $3$ in $[5]$ they prove that there is a one-to-one correspondence between local algebras of dimension $n+1$ and  additive structures on the projective space $\mathbb{P}^n$. Recall that an additive structure on an algebraic variety over a field $\mathbb{K}$ is an isomorphism class of faithful actions with an open orbit of the group $\mathbb{G}_a^n$, where $\mathbb{G}_a$ is the additive group of $\mathbb{K}$. The theory of additive structures is analogous to the theory of toric varieties over $\mathbb{K}$, that is, varieties which admit $\mathbb{G}_m^n$-action with an open orbit, $\mathbb{G}_m$ being the multiplicative group of $\mathbb{K}$. The additive case is less studied at the moment. However, the case of additive structures on projective hypersurfaces is treated in [1] and [2] using local algebras. 

The problem of the classification of local algebras has been studied by many people. For the algebras of dimension not greater than $6$, the  classification (excluding the case when characteristic of the field is $2$ or $3$) was obtained in the work of Mazzola, see [7]. In fact, Mazzola worked with nilpotent commutative associative algebras without unit, but these algebras of dimension $n$ clearly are in one-to-one correspondence with local algebras with unit of dimension $n+1$. He obtained the complete list of isomorphism classes (there is a  finite number of them in this case), using Hochschild cohomology theory.

A more elementary approach was used in the work of Poonen, see $[9]$. He provided direct computations which are valid for a ground field of any characteristic. 

Iarrobino studied local algebras extensively, see his work $[6]$.

Surpunenko (see $[12]$) showed that the number of pairwise non-isomorphic algebras is infinite in any dimension greater than $6$. He worked with matrix algebras to show that. See also $[13]$.

The results of the paper can be generalized in two ways. Firstly, one can consider local algebras with arbitrary $t_1$. However, the situation seems to be much more complicated in this case. Secondly, one can study non-graded local algebras. In our case, when $t_1=2$, one may hope to find infinite families of non-isomorphic algebras for some of the  sequences that we list in Theorem $1$. 

The author is grateful to Ivan Arzhantsev for posing the problem and encouraging in writing the paper, to Ilya Kazachkov and the University of Basque Country for their hospitality during the work on the second version of this paper, and to the referee for the careful reading of the paper and many valuable remarks.

%%%%%%%%%%%%%%%%%%%%%%%%%%%%%%%%%%%%%%%%%%%%%%%%%%%%%%%%%%%%%%%%%%%%%%%
%
\section{Preliminaries}
\label{sec1}

Fix a base field $\mathbb{K}$ which we assume to be algebraically closed of characteristic zero. 
We consider commutative associative local finite-dimensional algebras $R$ over $\mathbb{K}$ with identity and the maximal ideal $\mathfrak{m}$. The word ``algebra'' means this throughout the paper.

Let us make some preliminary remarks concerning the structure of such algebras. First of all, since $\mathbb{K}$ is algebraically closed and $R$ is finite-dimensional we have $R \slash \mathfrak{m} \simeq \mathbb{K}$. It is easy to see that the ideal $\mathfrak{m}$ is nilpotent. It is also well known that if the images of the elements $x_1, \dots, x_n$ span $\mathfrak{m} \slash \mathfrak{m}^2$ then these elements generate the ideal $\mathfrak{m}$ and the algebra $R$ as well, thanks to Nakayama's lemma. The minimal number of generators of $R$ is equal to the dimension of $\mathfrak{m} \slash \mathfrak{m}^2$. Moreover, the elements $x_1^{k_1}\dots x_n^{k_n}$ for $k_1+\dots+k_n=k$ span $\mathfrak{m}^k \slash \mathfrak{m}^{k+1}$. Recall the following definition.

\begin{definition}
The \textit{Hilbert-Samuel sequence} of an algebra $R$ is the sequence $\mathrm{T}=\mathrm{T}(R)=(t_0, t_1, \dots)$ with $t_i=t_i(R)=\dim_{\mathbb{K}} \mathfrak{m}^i \slash \mathfrak{m}^{i+1}$ where $\mathfrak{m}^0=R$. 
\end{definition}

Occasionally we say that an algebra $R$ is of the type $\mathrm{T}(R)$. Since $R$ is finite-dimensional, only a finite number of $t_i$ are non-zero, so we write only non-zero terms of the sequence. As we have mentioned above $\dim_{\mathbb{K}} R\slash \mathfrak{m} = t_0(R) = 1$. We introduce the following definition. 

\begin{definition}
A Hilbert-Samuel sequence $\mathrm{T}$ is called a \textit{sequence of finite type} if there is only a finite number of pairwise non-isomorphic algebras with this sequence. Otherwise, $\mathrm{T}$ is called a \textit{sequence of infinite type}.
\end{definition}

\begin{example}
Any sequence with $t_1=1$ is of finite type. Moreover, one can show that $t_1(R)=1$ implies $\mathrm{T}(R)=(1, 1, \dots, 1)$, that is $t_i=1$ for $0 \leq i \leq n-1$, $t_j=0$ for $j\geq n$ for some $n$, and $R\simeq \mathbb{K}[x]\slash (x^{n})$. 
\end{example}

We address the following problem: find all the Hilbert-Samuel sequences of finite type. We give a solution to this problem for graded algebras generated by two elements of degree $1$, that is $t_1=2$. This case is especially nice since we have the following proposition ($[6]$, Lemma $1.3$).

\begin{proposition} 
If $R$ is an algebra with $t_1(R)=2$ then its Hilbert-Samuel sequence has the form $(1, 2, 3, \dots, n, t_n, \dots, t_{n+k})$ where $t_i=i+1$ for $0 \leq i \leq n-1$ and $n = t_{n-1} \geq t_n \geq \dots \geq t_{n+k} \geq 0$ for some $n\geq 2, k\geq 0$. Conversely, for any sequence of this form there is an algebra with this sequence.
\end{proposition} 

The next definition is useful for us.

\begin{definition}
The \textit{associated graded algebra} of a local algebra $R$ with maximal ideal $\mathfrak{m}$ is the vector space $$\mathrm{Gr}(R) = \bigoplus_{k=0}^{\infty} \mathfrak{m}^k \slash \mathfrak{m}^{k+1}$$ with multiplication defined on the direct summands in the following way: let $\bar{u}\in \mathfrak{m}^k \slash \mathfrak{m}^{k+1}$, $\bar{v}\in \mathfrak{m}^l \slash \mathfrak{m}^{l+1}$ and $u\in \mathfrak{m}^k$, $v\in \mathfrak{m}^l$ be their representatives. Then we define $\bar{u}\bar{v}=\overline{uv}\in \mathfrak{m}^{k+l}\slash \mathfrak{m}^{k+l+1}$. 
\end{definition}

It is easy to see that $R$ and $\mathrm{Gr}(R)$ are isomorphic as vector spaces. In the following definition the standard grading on the ring of polynomials is assumed.

\begin{definition}
An ideal $I$ in $\mathbb{K}[x_1, \dots, x_n]$ is called \textit{homogeneous} if for any polynomial $f\in I$ all the homogeneous components of $f$ are contained in $I$.   
\end{definition}

It is easy to see that a local algebra is isomorphic to a quotient of the ring of polynomials by a homogeneous ideal if and only if it is isomorphic to its associated graded algebra. Such algebras are graded and generated by elements of degree $1$. 

We denote the space of homogeneous polynomials of degree $n$ in $\mathbb{K}[x,y]$ as $\mathbb{K}[x,y]_n$ and the $n$-th homogeneous component of a homogeneous ideal $I$ as $I_n$. We also write $(f)_n$ for the $n$-th component of an homogeneous ideal $(f)$ generated by a homogeneous polynomial $f$. 

\begin{definition}
A Hilbert-Samuel sequence $\mathrm{T}$ is called a \textit{sequence of homogeneous finite type} if there is only a finite number of pairwise non-isomorphic graded (and generated by elements of degree $1$) algebras with such a Hilbert-Samuel sequence. Otherwise, $\mathrm{T}$ is called a \textit{sequence homogeneous infinite type}.
\end{definition}

To solve the problem of finding all the Hilbert-Samuel sequences of finite type completely in the case $t_1=2$ one has to construct infinite families of non-graded algebras of the types listed in Table $1$ or to prove that there are none.
%%%%%%%%%%%%%%%%%%%%%%%%%%%%%%%%%%%%%%%%%%%%%%%%%%%%%%%%%%%%%%%%%%%%%%%
%
\section{Dimension counting}
\label{sec2}

\begin{proposition}
The sequences $\mathrm{T}$ in Table $1$ are precisely those for which the inequality $\mathrm{dim}\ \mathrm{G}_\mathrm{T} \leq 3$ holds. 
\end{proposition}

\begin{proof}
We classify the sequences $\mathrm{T} = (1, 2, \dots, n, t_n, \dots, t_{n+k})$ with $n = t_{n-1} \geq t_n \geq \dots \geq t_{n+k}$ according to the parameter $n$. It can be defined as the first index $i$ such that $t_i \neq i+1$.  Since in our case $t_1 = 2$ we have $n \geq 2$. We denote $\mathrm{dim}\ \mathrm{G}_\mathrm{T}$ as $d = d(\mathrm{T})$ for brevity. We compute $d(\mathrm{T})$ using Theorem $2$.

We first consider the case $n = 2$. Let us start with the subcase $e_2 = t_1 - t_2=2$. We use the graphic representations of the corresponding Hilbert-Samuel sequences. The following diagram corresponds to the sequence $(1, 2)$ which is a special case of $\mathrm{T}_1$:

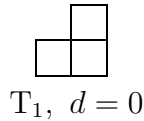
\begin{center}
\begin{minipage}{12em}
\begin{center}
\young(:\break,\break \break)

\
$\mathrm{T}_1,\ d=0$
\end{center}
\end{minipage}
\end{center}

\

\

If $e_2=1$, we have two sequences $\mathrm{T}_2=(1, 2, 1)$ and $(1, 2, 1, 1, \dots, 1)$ which is a special case of $\mathrm{T}_5$. They give different values of $d$:

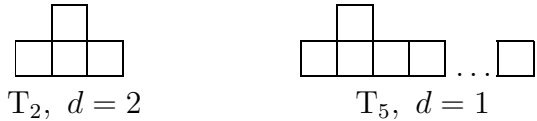
\begin{center}
\begin{minipage}{8em}
\begin{center}
\young(:\break:,\break \break \break)

\
$\mathrm{T}_2,\ d=2$
\end{center}
\end{minipage}
\hspace{1cm}
\begin{minipage}{8em}
\begin{center}
\young(:\break::,\break \break \break \break) \dots \young(\break)

\
$\mathrm{T}_5,\ d=1$
\end{center}
\end{minipage}
\end{center}

\

\

In what follows we do not write the sequence but only show the corresponding diagram. If $e_2=0$ we have three possibilities:

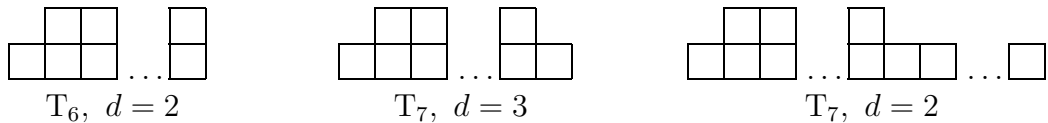
\begin{center}
\begin{minipage}{8em}
\begin{center}
\young(:\break \break,\break \break \break) \dots \young(\break,\break) 

\
$\mathrm{T}_6,\ d=2$
\end{center}
\end{minipage}
\hspace{1cm}
\begin{minipage}{8em}
\begin{center}
\young(:\break \break,\break \break \break) \dots \young(\break:,\break \break)

\
$\mathrm{T}_7,\ d=3$
\end{center}
\end{minipage}
\hspace{1cm}
\begin{minipage}{12em}
\begin{center}
\young(:\break \break,\break \break \break) \dots \young(\break::,\break \break \break) \dots
\young(\break)

\
$\mathrm{T}_7,\ d=2$
\end{center}
\end{minipage}
\end{center}

\

\

We see that in all these cases we have $d \leq 3$. Now we turn to sequences with $n=3$. We consider the case $n=3, e_3=3$:

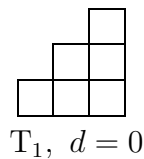
\begin{center}
\begin{minipage}{22em}
\begin{center}
\young(::\break,:\break\break,\break\break\break) 

\
$\mathrm{T}_1,\ d=0$
\end{center}
\end{minipage}
\end{center}

\

\

We turn to the case $n=3, e_3=2$. 

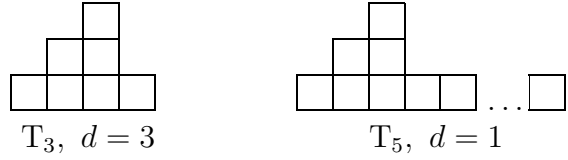
\begin{center}
\begin{minipage}{6em}
\begin{center}
\young(::\break:,:\break\break:,\break\break\break\break)

\
$\mathrm{T}_3,\ d=3$
\end{center}
\end{minipage}
\hspace{1cm}
\begin{minipage}{10em}
\begin{center}
\young(::\break::,:\break\break::,\break\break\break\break\break) \dots \young(\break)

\
$\mathrm{T}_5,\ d=1$
\end{center}
\end{minipage}
\end{center}

\

\

The case $n=3, e_3=1$:

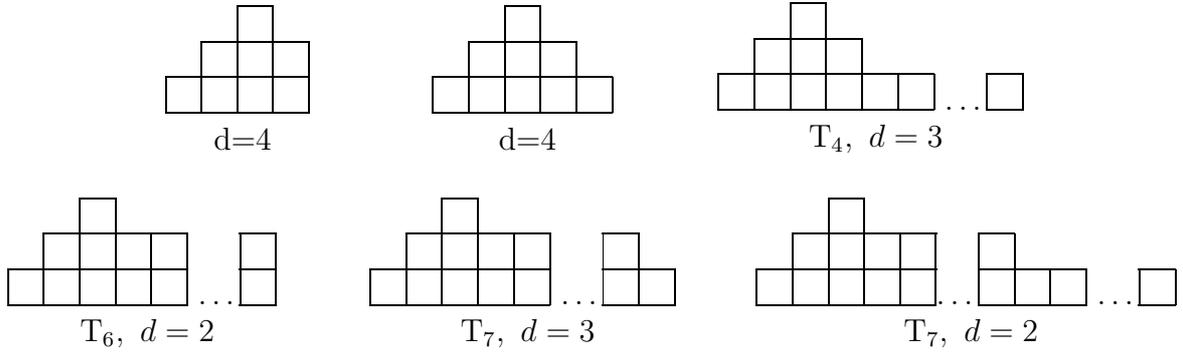
\begin{center}
\begin{minipage}{6em}
\begin{center}
\young(::\break:,:\break\break\break,\break\break\break\break)

\
d=4
\end{center}
\end{minipage}
\hspace{1cm}
\begin{minipage}{6em}
\begin{center}
\young(::\break::,:\break\break\break:,\break\break\break\break\break)

\
d=4
\end{center}
\end{minipage}
\hspace{1cm}
\begin{minipage}{10em}
\begin{center}
\young(::\break::::,:\break\break\break:::,\break\break\break\break\break\break) \dots \young(\break)

\
$\mathrm{T}_4,\ d=3$
\end{center}
\end{minipage}

\

\

\begin{minipage}{10em}
\begin{center}
\young(::\break::,:\break\break\break\break,\break\break\break\break\break) \dots \young(\break,\break) 

\
$\mathrm{T}_6,\ d=2$
\end{center}
\end{minipage}
\hspace{0.6cm}
\begin{minipage}{10em}
\begin{center}
\young(::\break::,:\break\break\break\break,\break\break\break\break\break) \dots \young(\break:,\break\break)

\
$\mathrm{T}_7,\ d=3$
\end{center}
\end{minipage}
\hspace{0.6cm}
\begin{minipage}{14em}
\begin{center}
\young(::\break::,:\break\break\break\break,\break\break\break\break\break)\dots \young(\break::,\break\break\break) \dots \young(\break)

\
$\mathrm{T}_7,\ d=2$
\end{center}
\end{minipage}
\end{center}

\

\

Consider the case $e_3=0$. We do not list all the possibilities, instead we notice that if $e_3=0$ then $$d = \sum_{j\geq 3} (e_j + 1) e_{j+1} \geq \sum_{j\geq 3} e_{j+1} = \sum_{j\geq 3} e_j = t_2 = n = 3.$$

It follows that if we want $d=3$ we must have $e_j = 0$ whenever $e_{j+1} > 0$. The inequality $d\leq 3$ yields the following diagrams:

\
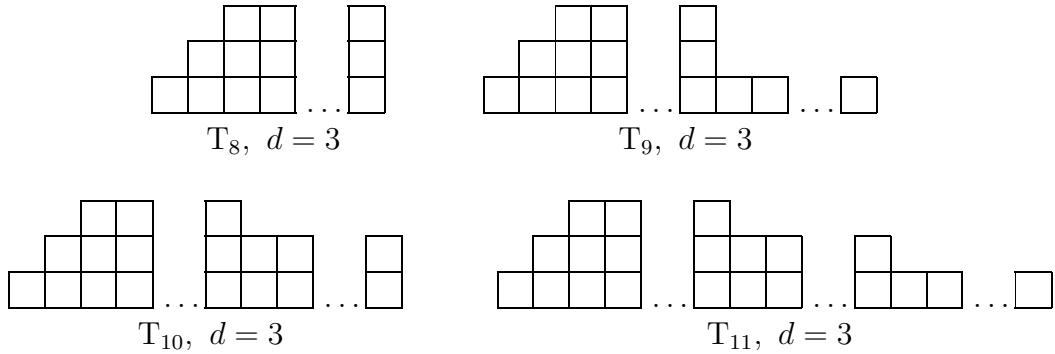
\begin{center}
\begin{minipage}{8em}
\begin{center}
\young(::\break\break,:\break\break\break,\break\break\break\break) \dots \young(\break,\break,\break)

\
$\mathrm{T}_8,\ d=3$
\end{center}
\end{minipage}
\hspace{0.6cm}
\begin{minipage}{14em}
\begin{center}
\young(::\break\break,:\break\break\break,\break\break\break\break) \dots \young(\break::,\break::,\break\break\break) \dots \young(\break)

\
$\mathrm{T}_9,\ d=3$
\end{center}
\end{minipage}

\

\

\begin{minipage}{14em}
\begin{center}
\young(::\break\break,:\break\break\break,\break\break\break\break) \dots \young(\break::,\break\break\break,\break\break\break) \dots \young(\break,\break)

\
$\mathrm{T}_{10},\ d=3$
\end{center}
\end{minipage}
\hspace{0.6cm}
\begin{minipage}{18em}
\begin{center}
\young(::\break\break,:\break\break\break,\break\break\break\break) \dots \young(\break::,\break\break\break,\break\break\break) \dots \young(\break::,\break\break\break) \dots \young(\break)

\
$\mathrm{T}_{11},\ d=3$
\end{center}
\end{minipage}

\end{center}

\

\

Now we are ready to deal with the general case $n\geq 4$. First of all, notice that $d \geq \sum_{j\geq n} e_{j+1}$, and $\sum_{j\geq n} e_j = t_{n-1} = n \geq 4$. Hence to obtain $d\leq 3$ we must have $e_n \geq n-3$, that is $t_n\leq 3$. 

We start with the case $t_n = 0$ in which case we get the sequence $\mathrm{T}_1$ with $d=0$.

If $t_n = 1$ then $e_n =n - 1 \geq 3$ and so to get $d \leq 3$ we must have $e_{n+1} = 0$, since $d\geq (e_{n} + 1) e_{n+1}$. We obtain the sequence $\mathrm{T}_5$. In this case $d=1$.

If $t_n = 2$ then $e_n = n - 2 \geq 2$ and again to get $d\leq 3$ we must have $e_{n+1} \leq 1$. If $e_{n+1} = 0$ we get two options: $T_6$ (in this case $d=2$) and $\mathrm{T}_7$ (in this case $d=3$ or $d=2$ depending on the number of $t_j = 1$). If $e_{n+1} = 1$ then $t_n=2, t_{n+1}=1$ and hence $d = (n - 2 + 1)\times 1 + \sum_{j\geq n+1} e_{j+1} \geq 3 + t_{n+1} = 4$.

Finally, if $t_n = 3$ then $e_n = n - 3 \geq 1$ and to get $d\leq 3$ again we get $e_{n+1} \leq 1$. If $e_{j+1} = 0$ we get the sequences $\mathrm{T}_{9}, \mathrm{T}_{10}, \mathrm{T}_{11}$ with $d=3$ in each case. If $e_{j+1} = 1$ we get $t_n =3, t_{n+1} = 2$. Hence $d\geq 2 + 2 = 4$. The proof is complete.
\end{proof}

\section{Sequences of homogeneous finite type}

In this section we prove that the sequences listed in Table $1$ are of homogeneous finite type.

\begin{proposition}
For any $n \geq 2$ there is only one algebra (up to isomorphism) of the type $$\mathrm{T_1}=(1, 2, \dots, n-1, n),$$ that is $t_i=i+1$ for $0 \leq i \leq n-1$, $t_j=0$ for $j\geq n$.
\end{proposition}
\begin{proof} 
Pick two generators $x, y$ of $\mathfrak{m}$ and notice that there are no non-trivial relations containing monomials in $x, y$ of degree less than $n$. On the other hand, any monomial in $x$ and $y$ of degree $n$ or more equals zero. Hence our algebra is isomorphic to $\mathbb{K}[x, y] \slash (x, y)^{n}$.
\end{proof}

The first part of the next proposition is a special case of a general classification result proved in $[11]$ by J. Sally. See also $[10]$.

\begin{proposition}
The sequences $$\mathrm{T}_2 = (1, 2, 1), \ \ \mathrm{T}_3 = (1, 2, 3, 1)$$ are of homogeneous finite type.
\end{proposition}
\begin{proof}
Consider an algebra $R\simeq \mathbb{K}[x, y] \slash I$ for some ideal $I$ with $\mathrm{T}(R)=T_2$. Up to a scalar there is only one non-zero element $z \in \mathfrak{m}^2$. We define a map $$\theta: \mathfrak{m} / \mathfrak{m}^2 \rightarrow \mathfrak{m}^2 = \langle z \rangle \simeq \mathbb{K}$$ in the following way: $$\theta (a\bar{x} + b\bar{y}) = (a x + b y)^2 = (a^2 u_0 + 2ab u_1 + b^2 u_2)z$$ where $\bar{x}, \bar{y}$ are images of $x, y$ in $\mathfrak{m} / \mathfrak{m}^2$, $x^2 = u_0 z, xy = u_1 z, y^2 = u_2 z$, and $a, b, u_0, u_1, u_2 \in \mathbb{K}$. Notice that $\theta$ is non-zero (since the powers of linear forms span $\mathbb{K}[x, y]$ for any $n\geq 1$) map that is a polynomial in $a, b$ which we can think of as homogeneous coordinates on $\mathbb{P}^1$. The roots of this polynomial are two points in $\mathbb{P}^1$. If they are different then after a linear change of coordinates we may assume them to be $[1:0]$ and $[0:1]$. If they coincide we may assume it is $[1:0]$. In the first case we get an algebra $$\mathbb{K}[x, y]/(x^2, y^2)+(x, y)^3$$ and in the second case we get $$\mathbb{K}[x, y]/(xy, y^2)+(x, y)^3.$$

The case $\mathrm{T}_3$ is analogous. We define a map $$\theta: \mathfrak{m} / \mathfrak{m}^2 \rightarrow \mathfrak{m}^3 = \langle z \rangle \simeq \mathbb{K}$$ as follows: $$\theta (a\bar{x} + b\bar{y}) = (a x + b y)^3 = (a^3 u_0 + 3a^2b u_1 + 3 a b^2 u_2 + b^3 u_3)z$$ where $x^3 = u_0 z, x^2 y = u_1 z, x y^2 = u_2 z, y^3 = u_3 z$, and $a, b, u_0, u_1, u_2, u_3 \in \mathbb{K}$. Hence we get a non-zero  polynomial map of degree $3$ in coordinates $a, b$. We have three possibilities: all three roots of this polynomial are different, two of them coincide or all of them coincide. Hence we get three isomorphism classes of algebras of type $\mathrm{T}_3$.
\end{proof}

The next lemma is crucial for this section. It demonstrates that if there are several equal numbers $t_n = t_{n+1} = \dots = t_{n+k}$ for $k\geq 1$ in $\mathrm{T}$ then the ideal $I$ where $R = \mathbb{K}[x, y] \slash I$ must be of special form.

\begin{lemma}
Let $R = \mathbb{K}[x, y]/I$ be a local Artinian algebra with $I$ homogeneous. Let the Hilbert-Samuel sequence of $R$ be $$\mathrm{T}=(1, 2, t_2, \dots, t_{n-1}, s, s, \dots, s, t_{n+k+1}, \dots)$$ where $t_2, \dots, t_{n-1}$ are arbitrary, $t_n = \dots = t_{n+k} = s$, for $k\geq 1$ and arbitrary $s$. Then there is a polynomial $h\in \mathbb{K}[x, y]_s$ such that $(h)_n = I_n$.

\end{lemma}

\begin{proof}
We have $\mathrm{dim}_{\mathbb{K}}\ I_n=n+1-s$. Suppose that all the polynomials in $I_n$ have exactly $m$ common linear factors, $0 \leq m\leq n$. Hence $$I_n=\langle h q_1, \dots, h q_{n+1-s} \rangle = h \langle q_1, \dots, q_{n+1-s} \rangle$$ where $h$ is the product of all the common linear factors and $q_i$ are polynomials of degree $n-m$ without common factors. It follows that $$\mathrm{dim}_{\mathbb{K}}\ I_n=n+1-s\leq n-m+1=\mathrm{dim}_{\mathbb{K}}\ \mathbb{K}[x, y]_{n-m}$$ and hence $m\leq s$.

On the other hand, since $t_{n+1}=s$ we have $\mathrm{dim}_{\mathbb{K}}\ I_{n+1}=n+2-s$, that is $\mathrm{dim}_{\mathbb{K}}\ I_{n+1}=\mathrm{dim}_{\mathbb{K}}\ I_{n}+1$. We can choose a basis $f_1, \dots, f_r$ for $I_n$ where $r=n+1-s$ such that $$f_i=x^{d_i} y^{n-d_i}+a_{i, d_i-1}x^{d_i-1} y^{n-d_i+1}+\dots+a_{i, 0} y^{n}$$ where $d_1> \dots > d_r$ for $1\leq i\leq r$ and $a_{i, j}\in \mathbb{K}$ are such that $a_{i, d_j}=0$ for $j > i$. Consider the following elements of $I_{n+1}$: $$xf_1, xf_2, \dots, xf_r, yf_r.$$ Since our basis elements have different $x$-degrees, these polynomials are linearly independent, and since $\mathrm{dim}_{\mathbb{K}}\ I_{n+1}=r+1$ it follows that $$\langle xf_1, xf_2, \dots, xf_r, yf_r \rangle = I_{n+1}.$$ 

Hence $yf_i$ for $1\leq i\leq r-1$ can be expressed as a linear combination of $xf_{i+1}, \dots, xf_r, yf_r$. But then, comparing the $x$-leading terms of these polynomials, we see that necessarily $d_i = d_{i+1} + 1$. Hence after linear change we may assume that our basis has the form (here $d = d_1$): $$f_1 = x^{d} y^{n - d} + a_{1, d-r} x^{d-r}y^{n-d+r} + \dots + a_{1, 0} y^n,$$ $$\dots$$ $$f_r = x^{d-r+1} y^{n - d + r - 1}  + a_{r, d-r} x^{d-r}y^{n-d+r} + \dots + a_{r, 0} y^n.$$

In what follows we show that $y f_1=x f_2$. Indeed, we know that for some $u_i \in \mathbb{K}$ $$yf_1=xf_2 + u_3 xf_3 + \dots + u_r x f_r + u_{r+1} y f_r.$$

But now substituting $f_i$ we have 
\begin{multline*}
x^{d} y^{n - d +1} + a_{1, d-r} x^{d-r}y^{n-d+r +1} + \dots + a_{1, 0} y^{n+1} = yf_1 =  \\ =  x^{d} y^{n-d+1} +u_3 x^{d-1} y^{n-d+2} + \dots + u_r x^{d-r+2} y^{n-d+r-1} + u_{r+1} x^{d-r+1}y^{n-d+r} + \\  + v_{r+2} x^{d-r} y^{n-d+r+1} + \dots + v_{d-2} y^{n+1}
\end{multline*}

for some $v_j \in \mathbb{K}$. Hence all the $u_i = 0$, and $yf_1=xf_2$ as claimed. Arguing as above we have $$yf_1=xf_2,\ yf_2=xf_3,\ \dots,\ yf_{r-1}=xf_r.$$ 

Hence the decomposition of $f_i$ into a product of linear forms differs from that of $f_1$ by at most $i-1$ linear forms, $2\leq i \leq r=n-s+1$. Hence all the $f_i$ have at least $n-(n-s)=s$ common linear factors, and then $m\geq s$. 

So we have $m=s$. Thus $\mathrm{dim}_{\mathbb{K}}\ I_n=\mathrm{dim}_{\mathbb{K}}\ \mathbb{K}[x, y]_{n-m}$ and $I_n=h \mathbb{K}[x, y]_{n-m} = (h)_n$ for $h\in \mathbb{K}[x, y]_s$ as claimed. 
\end{proof}

Let us apply this result to the remaining sequences.

\begin{proposition}
The sequence $$\mathrm{T}=(1, 2, \dots, n, s, s, \dots, s)$$ where $t_i=i+1$ for $0 \leq i \leq n-1$, $t_{n}=\dots=t_{n+k}=s$, $k\geq 1$ is of homogeneous finite type for $s\leq 3$, $n\geq2$. This sequence corresponds to the types $\mathrm{T}_5,\ \mathrm{T}_6,\  \mathrm{T}_8$.
\end{proposition}

\begin{proof}
Let $R = \mathbb{K}[x, y]/I$ be a local Artinian algebra with $I$ homogeneous and $\mathrm{T}(R) = T$. By the previous lemma we know that $I_n = (h)_n$ for some polynomial $h\in \mathbb{K}[x, y]_s$. Moreover, $I$ is generated by its $n$-th component $I_n$ as an ideal. Indeed, let $J = (I_n)$. Then $J \subseteq I$ and $n + 2 -s \leq \mathrm{dim}_{\mathbb{K}}\ J_{n+1} \leq  \mathrm{dim}_{\mathbb{K}}\ I_{n+1} = n + 2 -s$ and hence $J_{n+1} = I_{n+1}$. We proceed by induction and see that $I = J$.

Hence there is one-to-one correspondence between subspaces $I_n$ (that in this case determine the ideal $I$ and hence the algebra $R$) and polynomials $h$ of degree $s\leq 3$. Since up to change of coordinates there are finitely many such polynomials, we get that this sequence is of homogeneous finite type. 
\end{proof}

\begin{proposition}
The following sequences are of homogeneous finite type for $n\geq 3$:

$$\mathrm{T}_9=(1, 2, 3, \dots, n-1, n, 3, 3, \dots, 3, 1, 1, \dots, 1),$$ where $t_n=\dots=t_{n+k}=3,\ t_{n+k+1}=\dots=t_{n+k+l}=1,\ n\geq3,\ k\geq 1,\ l\geq 2;$

\
$$\mathrm{T}_{10}=(1, 2, 3, \dots, n-1, n, 3, 3, \dots, 3, 2, 2, \dots, 2),$$ where $t_n=\dots=t_{n+k}=3,\  t_{n+k+1}=\dots=t_{n+k+l}=2,\ n\geq2,\ k\geq 1,\ l\geq 2;$

\
$$\mathrm{T}_{11}=(1, 2, 3, \dots, n-1, n, 3, 3, \dots, 3, 2, 2, \dots, 2, 1, 1, \dots, 1),$$ where $t_n=\dots=t_{n+k}=3,\  t_{n+k+1} = \dots = t_{n+k+l}=2,\ t_{n+k+l+1}= \dots = t_{n+k+l+s}=1,$
\ 
$n\geq 3,\ k\geq 1,\ l\geq 2,\ s\geq 2.$
\end{proposition}

\begin{proof}
Let $R = \mathbb{K}[x, y]/I$ be a local Artinian algebra with $I$ homogeneous and $\mathrm{T}(R) = T_9$. By Lemma $3.3$, $I_n = (f)_n$ for some $f\in \mathbb{K}[x, y]_3$ and $I_{n+k+1} = (h)_{n+k+1}$ for some linear form $h$. Counting dimensions as in the proof of the previous proposition we see that $I_{n+i} = (f)_{n+i}$ for $0\leq i \leq k$ and $I_{n+k+i} = (h)_{n+k+i}$ for $i \geq 1$. Thus $f$ and $h$ define $I$ uniquely. We have $(f)_{n+k+1} \subseteq I_{n+k+1} = (h)_{n+k+1}$ and hence $h$ divides $f$. Since $\mathrm{deg}\ f = 3$ up to a change of coordinates there are finitely many possibilities to choose a pair $(f, h)$ such that $h$ divides $f$, and thus $\mathrm{T}_9$ is of homogeneous finite type.

The case $\mathrm{T}_{10}$ is analogous, with polynomial $h$ being of degree $2$.

In the case $\mathrm{T}_{11}$ arguing by analogy we have a triple $(f, g, h)$ where $f\in \mathbb{K}[x, y]_3$, $g\in \mathbb{K}[x, y]_2$, $f\in \mathbb{K}[x, y]_1$, $h$ divides $g$, $g$ divides $f$, and $I_n = (f)_n$, $I_{n+k+1} = (g)_{n+k+1}$, $I_{n+k+l+1} = (h)_{n+k+l+1}$. Again, as up to a change of coordinates there are finitely many choices of such a triple we get that $\mathrm{T}_9$ is of homogeneous finite type.

\end{proof}

\begin{proposition}
The sequence $$\mathrm{T}_4=(1, 2, 3, 2, 1, 1, \dots , 1),$$ where $t_3=2, t_4=\dots=t_{4+k}=1,\ k\geq 1$ is of homogeneous finite type.
\end{proposition}
\begin{proof}
Let $R = \mathbb{K}[x, y]/I$ be a local Artinian algebra with $I$ homogeneous and $\mathrm{T}(R) = T_4$. By Lemma $3.3$ we have $$I_4=h\langle q_1, q_2, q_3, q_4 \rangle \simeq \langle q_1, q_2, q_3, q_4 \rangle \subseteq \mathbb{K}[x, y]_3$$ where $h$ is a linear form and the $q_i$ are linearly independent forms with no common factors. Thus they span $\mathbb{K}[x, y]_3$. We may assume that $h=x$. Notice that $y I_3 \subseteq I_4$, thus each form in $I_3$ is a multiple of $x$. To choose a subspace $I_3$ in $\mathbb{K}[x, y]_3$ that consists of polynomials that have $x$ as a factor is the same as to choose two non-proportional polynomials $ax^2+bxy+cy^2$ and $dx^2+exy+fy^2$ with $a, b, c, d, e, f \in \mathbb{K}$.

If $a=d=0$ we have the subspace $\langle xy, y^2\rangle$. If $a\neq 0$ or $d\neq 0$ we can assume this subspace to be $\langle x^2+cy^2, xy+fy^2 \rangle$ or $\langle x^2+bxy, y^2 \rangle$. In the second case after rescaling $y$ we have $\langle x^2+xy, y^2 \rangle$. 

Now consider the first case. If $c= 0$ then after rescaling $y$ we have $\langle x^2, xy + y^2 \rangle$ in the case $f\neq 0$ and $\langle x^2, xy\rangle$ in the case $f=0$. If $c\neq0$ we may add the second polynomial multiplied by a constant to the first one to get the full square $(x+uy)^2$ for some $u\in \mathbb{K}$. We make a linear change: $x \rightarrow x + uy, y \rightarrow y$. Then we have $\langle x^2, y^2+wxy\rangle$, $w\in \mathbb{K}$. If $w\neq 0$ then after  rescaling we get $\langle x^2, x^2+xy \rangle$ if $w\neq 0$ and if $w=0$ we get $\langle x^2, y^2 \rangle$. Hence there are a finite number of isomorphism classes.
\end{proof}

\begin{proposition}
The sequence $$\mathrm{T}_7=(1, 2, 3, \dots, n, 2, 2, \dots, 2, 1, \dots, 1)$$ with $t_i=i+1,\ 0\leq i \leq n-1,\ t_n=\dots=t_{n+k}=2,\ t_{n+k+1}=\dots=t_{n+k+1+l}=1,\ n\geq 1,\ k\geq 1,\ l\geq 0$ is of homogeneous finite type.
\end{proposition}

\begin{proof}
Let $R = \mathbb{K}[x, y]/I$ be a local Artinian algebra with $I$ homogeneous and $\mathrm{T}(R) = T_7$. We have $I_n= (f)_n$ for $f \in \mathbb{K}[x, y]_2$ so we may assume first that $f$ has two distinct roots, $f=xy$. We have $t_{n+k}=2, t_{n+k+1}=1$ so $I_{n+k}=(f)_{n+k}$, $I_{n+k+1}=(f)_{n+k+1} + \langle h \rangle$ for some homogeneous polynomial $h$ of degree $n+k+1$. Since $$(f)_{n+k+1}=(xy)_{n+k+1}=\langle x^{n+k}y, x^{n+k-1}y^2, \dots, xy^{n+k}\rangle$$ we may assume that $h=ax^{n+k+1}+by^{n+k+1}$ for $a, b\in \mathbb{K}$. Rescaling $x$ and $y$ we get $h=x^{n+k+1}+y^{n+k+1}$ or (permuting $x$ and $y$ if necessary) $h=x^{n+k+1}$. 

Now assume $f=x^2$. In the notation as above we can assume that $h=axy^{n+k}+by^{n+k+1}$. Again by a linear change we can get either $h=xy^{n+k}+y^{n+k+1}$ or $h=xy^{n+k}$ or $h=y^{n+k+1}$.

Thus we have a finite number of possibilities, and hence the sequence $\mathrm{T}$ is of finite type.
\end{proof}

\section{Proof of the main result}

We start with the following lemma.

\begin{lemma}
If two graded Artinian quotients $R_1 \simeq \mathbb{K}[x, y] / I_1$ and $R_2 \simeq \mathbb{K}[x, y] / I_2$, where $I_1$ and $I_2$ are homogeneous ideals, are isomorphic then they are isomorphic as graded algebras, by which we mean an isomorphism that preserves grading on both algebras.  
\end{lemma}
\begin{proof}

Suppose that $\phi: R_1 \rightarrow R_2$ is a given isomorphism. We have $\phi(x) = ax + by + f$, $\phi(y) = cx + dy + g$ for some $a, b, c, d \in \mathbb{K}$ and $f, g \in \mathbb{K}[x, y]$ of degree two or more. Clearly, elements of the form $x^{k_1} y^{k_2}$ with $k_1+k_2=k$ generate $\mathfrak{m}_i^k \slash \mathfrak{m}_i^{k+1} \simeq \mathbb{K}[x, y]_k \slash (I_i)_k$ where $\mathfrak{m}_i$ is the maximal ideal of $R_i$, for $i=1, 2$, $k\geq0$. Notice that $\phi$ induces an isomorphism $\mathfrak{m}_1^k \slash \mathfrak{m}_1^{k+1} \simeq \mathfrak{m}_2^k \slash \mathfrak{m}_2^{k+1}$ for any $k \geq 0$ by sending $\bar{x}^{k_1} \bar{y}^{k_2}$ to $(a\bar{x} + b\bar{y})^{k_1} (c\bar{x} + d\bar{y})^{k_2}$. Consider a new map $\psi: R_1 \rightarrow R_2$ that sends $x$ to $\psi(x) = ax + by$, $y$ to $\psi(y) = cx + dy$ where $a,b,c,d$ are the same as above. This is a homomorphism of algebras since $\phi$ is a homomorphism, and the ideals $I_1$ and $I_2$ are homogeneous. Notice that $\psi$ also induces an isomorphism $\mathfrak{m}_1^k \slash \mathfrak{m}_1^{k+1} \simeq \mathfrak{m}_2^k \slash \mathfrak{m}_2^{k+1}$. Hence $\psi$ is surjective, and since $R_1$ and $R_2$ have the same dimension as algebras over $\mathbb{K}$, it is injective, and thus $\psi$ is an isomorphism.

\end{proof}

Now we are ready to prove the main result.

\begin{proof}[Proof of Theorem $1$]

Let $R_1 \simeq \mathbb{K}[x, y] / I_1$ and $R_2 \simeq \mathbb{K}[x, y] / I_2$ be local Artinian algebras with $I_1$ and $I_2$ homogeneous ideals. Lemma $4.1$ shows that if two graded algebras $R_1$ and $R_2$ are isomorphic then there is an isomorphism that preserves grading on both algebras. This isomorphism can be lifted to a linear automorphism of $\mathbb{K}[x, y]$ which sends $I_1$ to $I_2$. Hence we have an algebraic action of the group $\mathrm{PGL}(2)$ on the algebraic variety (for the proof that the Hilbert scheme is in fact an algebraic variety see $[6$, Theorem $2.9]$) $\mathrm{G_T}$ parametrizing homogeneous ideals that correspond to local algebras of the type $\mathrm{T}$. The orbits of this action correspond bijectively to isomorphism classes of algebras. The dimension of $\mathrm{G_T}$ is known from Theorem $2$.     

If $\mathrm{dim}\ G_{\mathrm{T}} > \mathrm{dim}\ \mathrm{PGL}(2) = 3$ then the number of the orbits of this action is infinite, and there are infinitely many non-isomorphic algebras of the type $\mathrm{T}$. Hence we have to deal with the case $\mathrm{dim}\ G_{\mathrm{T}} \leq 3$. All the sequences $\mathrm{T}$ that satisfy this inequality are listed in Table $1$ by Proposition $2.1$. Then in Propositions $3.1 - 3.7$ it is proved that these sequences are of homogeneous finite type, and the statement follows.    
\end{proof}

%%%%%%%%%%%%%%%%%%%%%%%%%%%%%%%%%%%%%%%%%%%%%%%%%%%%%%%%%%%%%%%%%%%%

\end{document}